\documentclass[12pt]{article}
\usepackage{amsthm,amstext, amsmath,latexsym,amsbsy,amssymb}
  
\newtheorem{proposition}{Proposition}[section]
\newtheorem{remark}{Remark}[section]
\newtheorem{lemma}{Lemma}[section] 

\numberwithin{equation}{section}
\newtheorem{theorem}{Theorem}[section] 

\newcommand{\h}{\hspace*{.24in}}

\allowdisplaybreaks
\begin{document}

\title{A Parallel Four Step Domain Decomposition Scheme for Coupled Forward Backward Stochastic Differential Equations}
\author{Minh-Binh TRAN\\
Laboratoire Analyse G\'eom\'etrie et Applications\\
Institut Galil\'ee, Universit\'e Paris 13, France\\
Email: binh@math.univ-paris13.fr}
\maketitle
\begin{abstract} Motivated by the idea of imposing paralleling computing on solving stochastic differential equations (SDEs), we introduce a new Domain Decomposition Scheme to solve forward-backward stochastic differential equations (FBSDEs) parallely. We reconstruct the Four Step Scheme in \cite{MaProtterYong:1994:SFB} with some different conditions and then associate it with the idea of Domain Decomposition Methods. We also introduce a new technique to prove the convergence of Domain Decomposition Methods for systems of quasilinear parabolic equations and use it to prove the convergence of our scheme for the FBSDEs. 
\end{abstract}
\section{Introduction}
\h The theory of forward-backward stochastic differential equations (FBSDEs) is a very active field of research since the first work of Pardoux and Peng \cite{PardouxPeng:1992:BSD} and Antonelli \cite{Antonelli:1993:BFS} came out in the early 1990s. These equations appear in a large number of application fields in stochastic control and financial mathematics. We refer to the monograph \cite{KarouiPengQuenez:1997:BSD}, \cite{Pham:2010:PDE} for details, further development and applications. Such systems strongly couple a forward stochastic differential equation with a backward one; and they can be written as a kind of stochastic two-point boundary value problems
\begin{equation}
\label{1e1}
\left \{ \begin{array}{ll}dX_t=b(t,X_t,Y_t)dt+\sigma(t,X_t,Y_t)dW_t,\vspace{.1in}\\ 
dY_t=-\Hat{b}(t,X_t,Y_t)dt-\Hat{\sigma}(t,X_t,Y_t)dW_t,\vspace{.1in}\\ 
X_0=x,Y_T=g(X_T).\end{array}\right. 
\end{equation}
Together with the theoretical studies on the systems (see \cite{Antonelli:1993:BFS}, \cite{Delarue:2002:OEU}, \cite{DelarueMenozzi:2006:AFB}, \cite{DelarueGuatteri:2006:WEU}, \cite{MaProtterYong:1994:SFB}, \cite{MaYong:1995:SFB}, \cite{MaYong:1999:FBS}, \cite{PardouxPeng:1992:BSD}), finding an efficient numerical scheme for FBSDEs has also become an important part of the theory. In order to solve a system of FBSDEs, we need to use the "decoupling PDE" technique, based on the so-called four step scheme (see \cite{MaProtterYong:1994:SFB}, \cite{MaYong:1995:SFB}, \cite{MaYong:1999:FBS}).  In which, the system of FBSDEs is associated with a quasilinear parabolic system of the following type
\begin{equation}
\begin{split}
\label{1e2}
\left \{ \begin{array}{ll}\frac{\partial \theta}{\partial t}+\sum_{i,j=1}^na_{i,j}\frac{\partial^2 \theta}{\partial x_i\partial x_j}+<{\nabla\theta},b(t,x,\theta)>+\hat{b}(t,x,\theta)=0,\mbox{ in } (0,T)\times\mathbb{R}^n,\vspace{.1in}\\ 
\theta(T,x)=g(x),\mbox{ on } \mathbb{R}^n,\end{array}\right. 
\end{split}
\end{equation}
where $\theta(t,x)$ is a vector of $m$ components $\theta=(\theta^1,\dots,\theta^m)$, $m\in\mathbb{N}$.
\\ From here, there are two directions to solve FBSDEs. The first trend is to solve FBSDEs by using the decoupling technique combining with some probability methods to avoid treating the PDEs directly (see \cite{BenderZhang:2008:TDM}, \cite{CvitanicZhang:2005:TSD}, \cite{DelarueMenozzi:2006:AFB}, \cite{MilsteinTretyakov:2006:NAF}). The second trend is to solve directly the PDEs. The first paper in this direction is the one of Douglas, Ma, Protter \cite{DouglasMaProtter:1996:NMF}, in which the PDE is treated by a finite difference method. Later in 2008, Ma, Shen and Zhao proposed a new approach based on the Hermite-spectral Method to treat the PDE (see \cite{MaShenZhao:2008:ONA}), which is then proved to be much more better than the previous one. In this paper, we present a new approach, still based on the second trend, to the coupled FDSDEs problem, by combining the classical Four Step Scheme with Domain Decomposition Methods or Schwarz Methods, with Waveform Relaxation. The idea is to impose Parallel Computing on solving SDEs numerically. We reconstruct the Four Step Scheme with some new conditions and then associate it with Schwarz Waveform Relaxation Methods to parallelize the system of quasilinear parabolic equations $(\ref{1e2})$: System $(\ref{1e2})$ is divided into $I$ subproblems, and each problem is solved seperatedly. The scheme is then proved to be well-posed and stable. Up to what we know, this is the first attempt trying to apply Domain Decomposition Algorithms to stochastic differential equations.
\\\h  In the pioneer work \cite{Lions:1987:OSA}, \cite{Lions:1989:OSA}, \cite{Lions:1990:OSA}, P. L. Lions laid the foundations of the modern theory of Schwarz Algorithms. With the development of parallel computers, the interest in Schwarz Methods have grown rapidly, as these methods lead to inherently parallel algorithms. However, the problem of convergence of Schwarz Methods still remains an open problem up to now. In his pioneer work \cite{Lions:1987:OSA}, \cite{Lions:1989:OSA}, \cite{Lions:1990:OSA}, P. L. Lions has proved that the classical Schwarz Method for Linear Laplace Equation is in fact equivalent to a sequence of projections in a Hilbert space. Moreover, he also observed that the Schwarz Sequences of linear elliptic equations is related to Minimum Methods over product spaces. This observation was used later by L. Badea in \cite{Badea:1991:OSA} to prove the convergence of the classical Schwarz Method for a class of linear elliptic equations. 
\\ Later, in \cite{GanderStuart:1998:STC} and \cite{Giladi:1997:STD}, M. Gander-A. Stuart and E. Giladi-H. B. Keller applied Schwarz Methods to the $1$-dimensional linear advection-diffusion equation. Refering to the paper \cite{Burrage:1996:PPW}, they call Schwarz Methods applied to parabolic equations by Schwarz Waveform Relaxation Algorithms. The techniques of proving the convergence used in these papers were Laplace and Fourier Transforms and some explicit calculations. An extension to the nonlinear reaction-diffusion equation in $1$-dimension was considered in \cite{Gander:1999:WRA}. With the hypothese $f'(c)\leq C$ in \cite{Gander:1999:WRA}, proofs of linear convergence on unbounded time domains, and superlinear convergence on finite time intervals were then given in case of $n$ subdomains, based on some explicit computations on the linearized equations. Another extension to monotone nonlinear PDEs in higher dimension was considered by Lui in \cite{Lui:2001:OSM}, \cite{Lui:2002:OLM}, \cite{Lui:2001:OMS}. The main idea of the papers is based on the well-known Sub-Super Solutions Method in the theory of partial differential equations and the initial guesses are usually sub or super solutions of the equations. Recently, an extension to Systems of Semilinear Reaction-Diffusion Equations was investigated in \cite{Descombes:2010:SWR}. This is the first paper trying to apply Schwarz Methods to a system of PDEs in $1$-dimension and the proof of convergence is based strongly on the technique introduced in \cite{Gander:1999:WRA}. 
\\ In order to solve FBSDEs by Schwarz Methods, we encounter the system of quasilinear parabolic equations $(\ref{1e2})$ in $n$-dimension. We then introduce a new technique, which allows us to study the convergence of Schwarz Algorithms for systems of nonlinear equations in $n$-dimension. 
\section{Forward-backward stochastic differential equations} 
The structure of this section is as follows: In Section 2.1, we will give the definition of forward-backward stochastic differential equations, then state some results on the existence and uniqueness of the equations; these results will be proved in Section 2.2.
\subsection{Existence and uniqueness results}
\h Let $\{W_t:t\geq 0\}$ be a d-dimensional Brownian motion defined on the probability space $(\Omega,\mathfrak{F},P)$. We define by $\{\mathfrak{F}_t\}$ the $\sigma$-field generated by $W$. We suppose that $\{\mathfrak{F}_t\}$ contains all the null sets of $\mathfrak{F}$ and consider the following forward-backward SDEs
\begin{equation}
\label{2e1}
\left \{ \begin{array}{ll}X_t=x+\int_0^tb(s,X_s,Y_s)ds+\int_0^t\sigma(s,X_s,Y_s)dW_s,\vspace{.1in}\\ 
Y_t=g(X_T)+\int_t^T\Hat{b}(s,X_s,Y_s)ds+\int_t^T\Hat{\sigma}(s,X_s,Y_s,Z_s)dW_s,\end{array}\right. 
\end{equation}
where $t$ belongs to $[0,T]$; the processes $X$, $Y$, $Z$ take values in $\mathbb{R}^n$, $\mathbb{R}^m$, $\mathbb{R}^{m\times d}$, respectively and $b$, $\Hat b$, $\sigma$, $\Hat{\sigma}$, $g$ take values in $\mathbb{R}^n$ ,$\mathbb{R}^m$, $\mathbb{R}^{n\times d}$, $\mathbb{R}^{m\times d}$ and $\mathbb{R}^m$, respectively. 
\\ Since we are only looking for {\it ordinary adapted solutions} of the FBSDEs $(\ref{2e1})$ (i.e. solutions which are $\{\mathfrak{F}_t\}$-adapted and square-integrable, and satisfy $(\ref{2e1})$ $P$-almost surely), we can write $(\ref{2e1})$ in the following form
\begin{equation}
\label{2e2}
\left \{ \begin{array}{ll}dX_t=b(t,X_t,Y_t)dt+\sigma(t,X_t,Y_t)dW_t,\vspace{.1in}\\ 
dY_t=-\Hat{b}(t,X_t,Y_t)dt-\Hat{\sigma}(t,X_t,Y_t,Z_t)dW_t,\vspace{.1in}\\ 
X_0=x,Y_T=g(X_T).\end{array}\right. 
\end{equation}
Now we state the conditions that we impose on $(\ref{2e1})$ and $(\ref{2e2})$:
\\ (A1) The functions $b$, $\Hat b$, $\sigma$, $\Hat\sigma$, $g$ are $C^1$-functions with bounded partial derivaties; and $g$ is bounded in $C^{2+\delta}(\mathbb{R}^m)$ for some $\delta$ in $(0,1)$. 
\\ (A2) The matrix $\sigma$ satisfies $$|\sigma(t,x,y)|\leq C,$$ and $$\sigma(t,x,y)\sigma^T(t,x,y)\geq \nu(|y|)I, \forall (t,x,y)\in[0,T]\times\mathbb{R}^n\times\mathbb{R}^m,$$ where $\nu$ is a positive continuous function, and $C$ is a positive constant. 
\\ (A3) We impose the following assumptions on $\Hat\sigma$:
\\ There exists a positive continuous function $\kappa$ such that
\begin{equation}
\label{2e3}
\mathop{\sup}\{~|z|~{\Large:}~~\Hat{\sigma}(t,x,y,z)=0\}\leq\kappa(|y|),\forall (t,x,y)\in [0,T]\times\mathbb{R}^n\times\mathbb{R}^m.
\end{equation}
\\ For fixed $(t,x,y)$, the function $z\mapsto \Hat\sigma(t,x,y,z) $ is bijective, $\Hat\sigma^{-1}(t,x,y)(\zeta)$ is contiuous with respect to $t,x,y$ and $\zeta$; and there exists a continuous function $\lambda$ from $\mathbb{R}$ to $\mathbb{R}_+$, such that 
\begin{equation}
\label{2e4}
|\Hat\sigma^{-1}(t,x,y)(\zeta)-\Hat\sigma^{-1}(t,x,y)(\zeta')|\leq \lambda(|y|)|\zeta-\zeta'|^{\alpha},\forall\zeta,\zeta'\in\mathbb{R}^{m\times n},
\end{equation}
where $\alpha$ is a constant lying in $[1,2)$.
\\ (A4) There exists a positive function $\eta$ and a positive constant $C$ such that for all $(t,x,y)$ in $[0,T]\times \mathbb{R}^n\times\mathbb{R}^m$, $$|b(t,x,y)|\leq \eta(|y|)$$ and $$\Hat{b}(t,x,0)\leq C.$$ 
\\ (A5) We suppose also that for all  $k$ $\in$ $\{1,\dots,m\}$ and for all $(t,x,y_1,\dots,$ $y_{k-1},y_{k+1},$ $\dots,y_{m})$ in $[0,T]\times\mathbb{R}^n\times\mathbb{R}^{m-1}$: 
$$\hat{b}^k(t,x,y_1,\dots,y_{k-1},0,y_{k+1},\dots,y_{m})=0,$$ and $\hat{b}^k(t,x,$ $y_1,\dots,$ $y_{k-1},$ $y_k,y_{k+1},$ $\dots,y_{m})$ is decreasing in $y_k$.
\\\h Assuming that $Y_t$ takes the form $\theta(t,X_t)$, $P$-almost surely, for all $t$ in $[0,T]$, by the It\^o's formula, we can transform the backward SDE in $(\ref{2e2})$ into the following system of PDEs
\begin{equation}
\label{2e7}
\left \{ \begin{array}{ll}\frac{\partial \theta}{\partial t}+\sum_{i,j=1}^na_{i,j}\frac{\partial^2 \theta}{\partial x_i\partial x_j}+<{\nabla\theta},b(t,x,\theta)>+\hat{b}(t,x,\theta)=0,\mbox{ in } (0,T)\times\mathbb{R}^n,\vspace{.1in}\\ 
\theta(T,x)=g(x),\mbox{ on } \mathbb{R}^n,\end{array}\right. 
\end{equation}
where we define $\sigma^T(t,x,\theta)$ to be the transposed matrix of $\sigma(t,x,\theta)$ and $ (a_{i,j})$ to be $\frac{1}{2}\sigma(t,x,\theta)\sigma^T(t,x,\theta)$.
The result on the existence and uniqueness of solutions of $(\ref{2e2})$ then follows
\begin{theorem}
\label{2t3}
(Existence and Uniqueness Theorem) Suppose that Conditions $(A1)$- $(A4)$ above are satified, then Equation $(\ref{2e2})$ admits a unique solution $(X,Y,Z)$ defined as follows: 
\\\h The process $X$ is the solution of the following forward SDE
\begin{eqnarray}
\label{2e9}
X_t=x+\int_{0}^tb(s,X_s,\theta(s,X_s))ds+\int_0^t\sigma(s,X_s,\theta(s,X_s))ds,
\end{eqnarray}
where $\theta$ is the unique solution of $(\ref{2e7})$. 
\\\h The processes $Y_t$, $Z_t$ are then $\theta(t,X_t)$ and $z(t,X_t,\theta(t,X_t),\nabla\theta(t,X_t))$, where $z$ is a smooth mapping from $[0,T]\times\mathbb{R}^n\times\mathbb{R}^m\times\mathbb{R}^{m\times n}$ to $\mathbb{R}^{m\times d}$ satisfying
 \begin{equation}
\label{2e6}
\xi\sigma(t,x,y)+\Hat{\sigma}(t,x,y,z(t,x,y,\xi))=0, \forall (t,x,y,\xi)\in[0,T]\times \mathbb{R}^n\times \mathbb{R}^m\times \mathbb{R}^{m\times n}.
\end{equation}
\end{theorem}
\begin{remark} Assumptions $(A1)$ - $(A4)$ is similar to Assumptions  $(A1)$ - $(A4)$ in \cite{MaProtterYong:1994:SFB}. However, in $(\ref{2e4})$, $\alpha$ can vary in $[1,2)$ while in the condition $(2.12)$ of $(A3)$ \cite{MaProtterYong:1994:SFB}, $\alpha=1$. $(A5)$ will be used later in the proof of convergence for the Parallel Four Steps Domain Decomposition Scheme.
\end{remark}
\subsection{Proof of existence and uniqueness results}
\h This subsection is devoted to the proof of the existence and uniqueness results stated in Section $\ref{2t3}$. We first consider the algebraic equation $(\ref{2e6})$. The following result shows that $(\ref{2e6})$ has a solution.
\begin{proposition}
\label{2t1}
Under Assumption $(A3)$, Equation $(\ref{2e6})$ has a unique continuous solution $z$, that satisfies the estimate
$$|z(t,x,y,\xi)|\leq \lambda(|y|)|\xi||\sigma(t,x,y)|^{\alpha}+\kappa(|y|), ~~\forall (t,x,y,\xi)\in[0,T]\times \mathbb{R}^n\times\mathbb{R}^m\times \mathbb{R}^{m\times n}.$$
\end{proposition}
\begin{proof}
Since $z(t,x,y,\xi)$ is equal to $\Hat\sigma^{-1}(t,x,y)(\xi\sigma(t,x,y))$ and $\Hat\sigma^{-1}$, $\sigma$ are continuous, then $z(t,x,y,\xi)$ is continuous. Finally, Conditions $(\ref{2e3})$ and $(\ref{2e4})$ give
\begin{eqnarray*}
|z(t,x,y,\xi)|&\leq&|\Hat\sigma^{-1}(t,x,y)(\xi\sigma(t,x,y))-\Hat\sigma^{-1}(t,x,y)(0)|+|\Hat\sigma^{-1}(t,x,y)(0)|\\\nonumber
&\leq&\lambda(|y|)|\xi\sigma(t,x,y)|^{\alpha}+|\kappa(|y|)|.
\end{eqnarray*}
\end{proof}
\h Now, the following Proposition states a result on the existence and uniqueness of solutions to $(\ref{2e7})$:
\begin{proposition}
\label{2t2}
Suppose that $(A1)-(A4)$ hold. Then  the system $(\ref{2e7})$ admits a unique classical solution $\theta(t,x)$, such that $\theta(t,x)$, $\frac{\partial}{\partial t}\theta(t,x)$, $\nabla\theta(t,x)$, $\Delta\theta(t,x)$ are bounded in $C((0,T)\times\mathbb{R}^n)$.
\end{proposition}
\begin{proof} We first recall a useful result. Let $\omega$ be a bounded and smooth enough domain of $\mathbb{R}^n$, we consider the system
\begin{equation}
\label{3e1}
\left \{ \begin{array}{ll}\frac{\partial \phi}{\partial t}+\sum_{i,j=1}^na_{i,j}\frac{\partial^2  \phi}{\partial x_i\partial x_j}+<{\nabla \phi},b(t,x, \phi)>+\hat{b}(t,x, \phi)=0,\mbox{ in } (0,T)\times\omega,\vspace{.1in}\\ 
 \phi(t,x)=g(t,x),\mbox{ on } [0,T]\times\partial\omega, \vspace{.1in}\\ 
 \phi(T,x)=g(T,x),\mbox{ on } \omega.\end{array}\right. 
\end{equation}
Similarly as in Theorem $7.1$, Chapter VII of \cite{LadyzenskajaSolonnikovUralceva:1967:LQE} and Lemma 3.2 of \cite{MaProtterYong:1994:SFB}, we have the following Lemma
\begin{lemma}
\label{3l1}
Suppose that all the functions $a_{ij}$, $b_i$, $\Hat{b}$ are smooth, $g$ is bounded in $C^{1+\delta,2+\delta}([0,T]\times\omega)$ with $\delta$ belongs to $(0,1)$; and for all $(t,x,y)\in [0,t]\times\mathbb{R}^n\times\mathbb{R}^m$, we have
\begin{equation}
\label{3e2}
\nu_1(|y|)I\leq (a_{ij}(t,x,y))\leq \nu_2(|y|)I,
\end{equation}
\begin{equation}
\label{3e3}
|b(t,x,y)|\leq\mu(|y|),
\end{equation}
\begin{equation}
\label{3e4}
\left|\frac{\partial}{\partial x_l}a_{ij}(t,x,y)\right|+\left|\frac{\partial}{\partial y_k}a_{ij}(t,x,y)\right|\leq\mu(|y|),
\end{equation}
for some continuous positive functions $\nu_1(.)$, $\nu_2(.)$, $\mu(.)$; and
\begin{equation}
\label{3e5}
|\Hat{b}(t,x,y)|\leq C_1(1+|y|),
\end{equation}
\begin{equation}
\label{3e6}
|<\Hat{b}(t,x,y),y>|\leq C_2(1+|y|^2),
\end{equation}
for some positive constants $C_1$, $C_2$. Then $(\ref{3e1})$ admits a unique classical solution $\phi(t,x)$. Moreover, $\phi(t,x)$, $\frac{\partial}{\partial t}\phi(t,x)$, $\nabla\phi(t,x)$, $\Delta\phi(t,x)$ are bounded in $C((0,T)\times\mathbb{R}^n)$ by a constant which does not depend on $\omega$, and there exists a positive number $\delta'$ in $(0,1)$ such that $\phi$ belongs to $C^{1+\delta',2+\delta'}((0,T)\times\omega)$.
\end{lemma}
\h Now, we will apply this Lemma to our case. First of all, we verify that the conditions of Lemma $\ref{3l1}$ hold. We can see that $(\ref{3e2})$ is a consequence of $(A2)$, $(\ref{3e3})$ is a consequence of $(A4)$ and $(\ref{3e4})$ is a consequence of $(A1)$. We only need to prove $(\ref{3e5})$ and $(\ref{3e6})$. 
\\ Conditions $(A1)$ and $(A4)$ infer that
$$|\Hat{b}(t,x,y)|\leq|\Hat{b}(t,x,y)-\Hat{b}(t,x,0)|+|\Hat{b}(t,x,0)|\leq C_1(1+|y|),$$
and this implies
$$|<\Hat{b}(t,x,y),y>|\leq C_1'(1+|y|)|y|\leq C_2(1+|y|^2),$$
where $C_1'$ is a positive constant.
Lemma $\ref{3l1}$ then implies that there exists a solution $\phi(t,x)$ for all $\omega$ in $\mathbb{R}^n$ bounded and smooth enough. 
\\\h By a convergence argument similar as in \cite{MaProtterYong:1994:SFB}, we deduce that $(\ref{2e7})$ admits a unique classical solution $\theta(t,x)$, such that $\theta(t,x)$, $\frac{\partial}{\partial t}\theta(t,x)$, $\nabla\theta(t,x)$, $\Delta\theta(t,x)$ are bounded in $C((0,T)\times\mathbb{R}^n)$ .
\end{proof}
\h We consider the forward SDE on $X$ from $(\ref{2e2})$, with the assumption that $Y_t$ can be written under the form $\theta(t,X_t)$
\begin{equation}
\label{2e5}
\left \{ \begin{array}{ll}dX_t=b(t,X_t,\theta(t,X_t))dt+\sigma(t,X_t,\theta(t,X_t))dW_t,\vspace{.1in}\\ 
X_0=x.\end{array}\right. 
\end{equation}
From the Lipschitz condition $(A1)$, we can conclude that $(\ref{2e5})$ has a unique solution $X$, which belongs to $\mathbb{L}^2(0,T)$. We then have the following proof of Theorem ${\ref{2t3}}$, based on the previous two Propositions.
\begin{proof}
The proof is divided into two steps
{\\\it Step 1: $(X_t,Y_t,Z_t)$ is a solution of $(\ref{2e2})$.}
\\\h We note that the existence and uniqueness of $X_t$ has already been discussed in the previous paragraph, we only need to prove that $(Y_t,Z_t)$ is a solution of the backward SDE in the system $(\ref{2e2})$. It follows from It\^o's formula that 
\begin{eqnarray}
\label{3e7}
dY_t& = & \left[\frac{\partial\theta}{\partial t}+\sum_{i=1}^n\frac{\partial\theta}{\partial x_i}b^i(x,X_t,Y_t)+\sum_{i,j=1}^n\frac{\partial^2\theta}{\partial x_i\partial x_j}a_{ij}\right]dt\\\nonumber
& & +\sum_{i=1}^n\frac{\partial\theta}{\partial x_i}\sigma^i(t,X_t,Y_t)dW_t.
\end{eqnarray} 
Equations $(\ref{3e7})$, Proposition $\ref{2t1}$ and the fact that $\theta$ is a solution of $(\ref{2e7})$ lead to 
\begin{eqnarray}
\label{3e9}
dY_t &=& -\Hat{b}(t,X_t,Y_t)dt-\Hat\sigma(t,X_t,Y_t,Z_t)dW_t.
\end{eqnarray}
 Hence, $(X_t,Y_t,Z_t)$ is a solution of $(\ref{2e2})$.
{\\\it Step 2: $(X_t,Y_t,Z_t)$ is unique.}
\\\h Let $(X_t^*,Y^*_t,Z^*_t)$ a solution of $(\ref{2e2})$, and set $Y'_t$ to be $\theta(t,X_t^*)$, $Z'_t$ to be $z(t,X_t^*,$ $\theta(t,X_t^*),\nabla\theta(t,X_t^*))$. It suffices to show that $Y'_t=Y^*_t$ and $Z'_t=Z^*_t$.
\\\h We first try to get some estimate for the quantity $E(|Y'_t-Y^*_t|^2)$. It\^o's formula gives
\begin{eqnarray*}
d(Y'_t-Y^*_t)&=&\left[\frac{\partial\theta}{\partial t}+\sum_{i=1}^n\frac{\partial \theta}{\partial x_i}b^i(t,X_t^*,Y_t^*)+\sum_{i,j=1}^na_{ij}\frac{\partial^2\theta}{\partial x_i\partial x_j}+\Hat{b}(t,X_t^*,Y_t^*)\right]dt\\
&+ &\left[\sum_{i=1}^n\frac{\partial\theta}{\partial x_i}\sigma^t(t,X_t^*,Y_t^*)+\Hat\sigma(t,X_t^*,Y_t^*,Z_t^*)\right]dW_t.
\end{eqnarray*}
 By using It\^o's formula for the $k$-th component of $Y'_t-Y^*_t$, $k\in\{1,\dots,m\}$, we obtain
\begin{eqnarray}
\label{3e10}
& &d(Y'^k_t-{Y^*_t}^k)^2\\\nonumber
&=&\left[\frac{\partial\theta^k}{\partial t}+\sum_{i=1}^n\frac{\partial \theta^k}{\partial x_i}b^i(t,X_t^*,Y_t^*)+\sum_{i,j=1}^na_{ij}\frac{\partial^2\theta^k}{\partial x_i\partial x_j}+\Hat{b}^k(t,X_t^*,Y_t^*)\right]\times\\\nonumber
& &\times2(Y'^k_t-{Y^*_t}^k)dt+\\\nonumber
& &+2(Y'^k_t-{Y^*_t}^k)\left[\sum_{i=1}^n\frac{\partial\theta^k}{\partial x_i}\sigma^i(t,X_t^*,Y_t^*)+\Hat\sigma^k(t,X_t^*,Y_t^*,Z_t^*)\right]dW_t\\\nonumber
& &+\left[\sum_{i=1}^n\frac{\partial\theta^k}{\partial x_i}\sigma^i(t,X_t^*,Y_t^*)+\Hat\sigma^k(t,X_t^*,Y_t^*,Z_t^*)\right]^2dt.
\end{eqnarray}
From the previous equation, the fact that $\theta$ is a solution of $(\ref{2e7})$ implies
\begin{eqnarray}
\label{3e12}
& &d(Y'^k_t-{Y^*_t}^k)^2\\\nonumber
&=&2(Y'^k_t-{Y^*_t}^k)\{\sum_{i=1}^n\frac{\partial \theta^k}{\partial x_i}[b^i(t,X_t^*,Y_t^*)-b^i(t,X_t^*,Y'_t)]+\\\nonumber
& &+\sum_{i,j=1}^n\frac{\partial^2\theta^k}{\partial x_i\partial x_j}[a_{ij}(x,X_t^*,Y_t^*)-a_{ij}(x,X_t^*,Y'_t)]+\\\nonumber
& &+[\Hat{b}^k(t,X_t^*,Y_t^*)-\Hat{b}^k(t,X^*_t,Y_t')]\}dt+\\\nonumber
& &+2(Y'^k_t-{Y^*}^k_t)\left[\sum_{i=1}^n\frac{\partial\theta^k}{\partial x_i}\sigma^i(t,X_t^*,Y_t^*)+\Hat\sigma^k(t,X_t^*,Y_t^*,Z_t^*)\right]dW_t\\\nonumber
& &+\left[\sum_{i=1}^n\frac{\partial\theta^k}{\partial x_i}\sigma^i(t,X_t^*,Y_t^*)+\Hat\sigma^k(t,X_t^*,Y_t^*,Z_t^*)\right]^2dt.
\end{eqnarray}
Since $z$ is a solution of $(\ref{2e6})$, we have that
\begin{eqnarray}
\label{3e13}
& & \sum_{i=1}^n\frac{\partial\theta^k}{\partial x_i}\sigma^i(t,X_t^*,Y_t^*)+\Hat\sigma^k(t,X_t^*,Y_t^*,Z_t^*)\\\nonumber
&=&\sum_{i=1}^n\frac{\partial\theta^k}{\partial x_i}[\sigma^i(t,X_t^*,Y_t^*)-\sigma^i(t,X_t^*,Y'_t)]+\\\nonumber
& &+[\Hat\sigma^k(t,X_t,Y_t^*,Z_t^*)-\Hat\sigma^k(t,X_t^*,Y_t',Z_t')].
\end{eqnarray}
Since $Y^*_T$ is equal to $Y'_T$, then $E(|Y'^k_T-{Y^*_T}^k|^2)=0$. We can infer from $(\ref{3e12})$ and $(\ref{3e13})$ that
\begin{eqnarray}
\label{3e14}
& &E|Y'^k_t-{Y^*_t}^k|^2\\\nonumber
&=&-E\int_{t}^T2(Y^k_s-{Y^*_s}^k)\{\sum_{i=1}^n\frac{\partial \theta^k}{\partial x_i}[b^i(s,X_s^*,Y_s^*)-b^i(s,X_s^*,Y'_s)]+\\\nonumber
& &+\sum_{i,j=1}^n\frac{\partial^2\theta^k}{\partial x_i\partial x_j}[a_{ij}(x,X_s^*,Y_s^*)-a_{ij}(x,X_s^*,Y_s')]+\\\nonumber
& &+[\Hat{b}^k(s,X_s^*,Y_s^*)-\Hat{b}^k(s,X_s^*,Y_s')]\}ds+\\\nonumber
& &-E\int_t^T\{\sum_{i=1}^n\frac{\partial\theta^k}{\partial x_i}[\sigma^i(s,X_s^*,Y_s^*)-\sigma^i(s,X_s^*,Y_s')]+\\\nonumber
& &+[\Hat\sigma^k(s,X_s^*,Y_s^*,Z_s^*)-\Hat\sigma^k(s,X_s^*,Y_s',Z_s')]\}^2ds.
\end{eqnarray}
Using (A1), we deduce from $(\ref{3e14})$ that there exists a positive constant $M_1$ such that
\begin{eqnarray}\nonumber
\label{3e15}
& &E(|Y'^k_t-{Y^*_t}^k|^2)+E\int_t^T[\Hat\sigma^k(s,X_s^*,Y_s^*,Z_s^*)-\Hat\sigma^k(t,X_s^*,Y_s^*,Z_s')]^2ds\\
&\leq &M_1E\int_t^T|Y_s'^k-{Y^*_s}^k|(|Y'_s-Y^*_s|+|Z'_s-Z^*_s|)ds,
\end{eqnarray}
which leads to
\begin{eqnarray}\nonumber
\label{3e16}
& &E(|Y'_t-Y^*_t|^2)+E\int_t^T|\Hat\sigma(s,X_s^*,Y_s^*,Z_s^*)-\Hat\sigma(s,X_s^*,Y_s^*,Z_s')|^2ds\\
&\leq &M_2E\int_t^T|Y'_s-Y^*_s|(|Y'_s-Y^*_s|+|Z'_s-Z^*_s|)ds,
\end{eqnarray}
where $M_2$ is a positive constant.
Condition $(\ref{2e4})$ implies that
\begin{eqnarray*}
|Z_t'-Z_t^*|&=&|\Hat\sigma^{-1}\Hat\sigma(t,X_t^*,Y_t^*,Z_t^*)-\Hat\sigma^{-1}\Hat\sigma(t,X_t^*,Y_t^*,Z'_t)|\\
& \leq &\lambda(|Y_t|)|\Hat\sigma(t,X_t^*,Y_t^*,Z_t^*)-\Hat\sigma(t,X_t^*,Y_t^*,Z_t')|^\alpha\\
& \leq &M_3|\Hat\sigma(t,X_t^*,Y_t^*,Z_t^*)-\Hat\sigma(t,X_t^*,Y_t^*,Z'_t)|^\alpha,
\end{eqnarray*}
where $M_3$ is a positive constant, since $\theta$ is uniformly bounded . 
\\Therefore
\begin{eqnarray}
\label{3e17}
& &E(|Y'_t-{Y^*_t}|^2)+M_4E\int_t^T|Z'_s-Z_s^*|^\frac{2}{\alpha}ds\\\nonumber
&\leq &M_2E\int_t^T|Y_s'-{Y^*_s}|(|Y_s'-{Y^*_s}|+|Z_s'-{Z^*_s}|)ds\\\nonumber
&\leq &M_2E\int_t^T|Y_s'-{Y^*_s}|^2ds+M_2E\int_t^T|Y_s'-{Y^*_s}||Z_s'-{Z^*_s}|ds\\\nonumber
&\leq &M_2E\int_t^T|Y_s'-{Y^*_s}|^2ds+\frac{M_5}{\epsilon^{\frac{2}{2-\alpha}}}E\int_t^T|Y_s'-{Y^*_s}|^{\frac{2}{2-\alpha}}ds\\\nonumber
& &+\epsilon^\frac{2}{\alpha} M_{6}E\int_t^T|Z_s'-Z^*_s|^\frac{2}{\alpha}ds.
\end{eqnarray}
where $M_4$, $M_5$, $M_6$ are positive constants and $\epsilon$ is a small constant to be chosen. Notice that the last inequality comes from Young's inequality. Choosing $\epsilon$ small enough, we get from $(\ref{3e17})$ that 
\begin{eqnarray}
\label{3e18}
E(|Y_t'-Y^*_t|^2)&\leq &M_7\int_t^T[E(|Y_s'-Y^*_s|^2)+E(|Y_s'-Y^*_s|^2)^{\frac{1}{2-\alpha}}]ds,
\end{eqnarray}
where $M_7$ is positive constant. 
\\\h Secondly, we will prove that $E(|Y_t'-Y^*_t|^2)$ is bounded and then deduce that $E(|Y_t'-Y^*_t|^2)$ is equal to $0$. Setting $G(t)$ to be $E(|Y_t'-Y^*_t|^2)$, we will establish a {\it New Gronwall Inequality}, to prove that $G(t)$ is bounded. We have the following inequality on $G$
\begin{eqnarray*}
G(t)&\leq &M_7\int_t^T[G(s)+G(s)^{\frac{1}{2-\alpha}}+1]ds.
\end{eqnarray*}
Now denoting $M_7\int_t^T[G(s)+G(s)^{\frac{1}{2-\alpha}}+1]ds$ by  $H(t)$, we deduce that $G(t)$ is less than or equal to $H(t)$; and as a result $H'(t)$ is bounded by $M_7[H(t)+H(t)^{\frac{1}{2-\alpha}}+1]$. This leads to 
\begin{eqnarray*}
\frac{H'(t)}{H(t)+H(t)^{\frac{1}{2-\alpha}}+1}&\leq &M_7.
\end{eqnarray*}
We define 
\begin{eqnarray*}
K(r)=\int_0^r\frac{1}{\rho+\rho^{\frac{1}{2-\alpha}}+1}d\rho.
\end{eqnarray*}
Since $(K(H(t)))'$ is bounded by $M_7$, then $$K(H(t))\leq M_{7}T+K(H(0))=M_8.$$ In addition, we know that $K$ is increasing, which means that $H(t)$ is less than or equal to $K^{-1}(M_8)$. We then infer that $G(t)$ is bounded by a postive constant $M_9$. Defining $M_9^{-1}G(t)$ by $P(t)$, we then deduce that $P(t)$ is positive and bounded by $1$. This leads to 
\begin{eqnarray*}
P(t)&\leq &M_{10}\int_t^T[P(s)+P(s)^{\frac{1}{2-\alpha}}]ds\leq M_{11}\int_t^TP(s)ds ,
\end{eqnarray*}
where $M_{10}$, $M_{11}$ are positive constants; notice that $1\leq\alpha<2$.
\\ By the classical Gronwall's Lemma, $P(t)$ is equal to $0$, which implies that $Y'_t$ is coincided with $Y^*_t$. From this, we infer that $Z'_t$ is coincided with $Z^*_t$.
\\\h Now, since every solution $(X_t^*,Y_t^*,Z_t^*)$ of $(\ref{2e2})$ can be written as $(X_t^*,$$\theta(t,X_t^*),$ $z(t,X_t^*,\theta(t,X_t^*)))$ and since $(\ref{2e5})$ has only one solution, we can conclude that $(\ref{2e2})$ has a unique solution.
\end{proof}
\section{The Parallel Four Step Domain Decomposition Scheme}
This section is devoted to the construction of the Parallel Four Step Domain Decomposition Scheme and its proof of well-posedness and stability.
\subsection{Definition of the Scheme}
We now define a new Parallel Four Step Domain Decomposition Scheme, based on the results obtained in Section 2. 
\\ {\it Step 1:} Find a smooth mapping $z$ satisfying $(\ref{2e6})$.
\\ {\it Step 2:} Choose $l$ to be a constant large enough and consider the domain $\mathcal{O}_l=(-l,l)^n$. On $\mathcal{O}_l$, consider the following problem instead of $(\ref{2e7})$
\begin{equation}
\label{2e8}
\left \{ \begin{array}{ll}\frac{\partial \theta^l}{\partial t}+\sum_{i,j=1}^na_{i,j}\frac{\partial^2 \theta^l}{\partial x_i\partial x_j}+<{\nabla\theta^l},b(t,x,\theta^l)>+\hat{b}(t,x,\theta^l)=0,\mbox{ in } (0,T)\times\mathcal{O}_l,\vspace{.1in}\\ 
\theta^l(t,x)=g(x),\mbox{ on } (0,T)\times\partial\mathcal{O}_l,\vspace{.1in}\\  
\theta^l(T,x)=g(x),\mbox{ on } \mathcal{O}_l.\end{array}\right. 
\end{equation}
The same arguments as in the proof of Proposition $\ref{2t2}$ show that $(\ref{2e8})$ has a unique classical solution $\theta^l(t,x)$, where $\theta^l(t,x)$, $\frac{\partial}{\partial t}\theta^l(t,x)$, $\nabla \theta^l(t,x)$, $\Delta\theta^l(t,x)$ are bounded. Suppose that $\theta^l(t,x)=g(x)$ on $(0,T)\times(\mathbb{R}^n\backslash\mathcal{O}_l)$, then these arguments also show that  
\begin{eqnarray*}
\mathop{\lim}_{l\to\infty}||\theta^l-\theta||_{(L^{\infty}((0,T)\times\mathbb{R}^{n}))^m}=0.
\end{eqnarray*}
\\ {\it Step 3:} Solve the equation $(\ref{2e8})$ iteratively in the following manner
\begin{itemize}
\item Divide $\mathcal{O}_l$ into $I$ subdomains $$\mathcal{O}_l=\bigcup_{p=1}^I\Omega_p=\bigcup_{p=1}^I(-l,l)^{n-1}\times (a_p,b_p),$$ where $-l=a_1<a_2<b_1<\dots<a_I<b_{I-1}<b_I=l.$ Denote that $S_i=b_{i}-a_{i+1}$ for $i$ belongs to $\{1,\dots,I-1\}$ and $L_i=b_i-a_i$ for $i$ belongs to $\{1,\dots,I\}$. 
\item Choose a bounded initial guess $\theta^l_0$ in $\mathbb{C}^{\infty}(\mathbb{R}^n)$ at step $0$. Associate each subdomain $\Omega_p$ with a function $\theta_{p,0}^l$ such that $\theta^l_{p,0}=\theta^l_0$ on $\Omega_p$.
\item Solve the following $p$-th subproblem at iteration $\#q$ 
\begin{equation}
\label{2e10}
\left \{ \begin{array}{ll}\frac{\partial \theta^l_{p,q}}{\partial t}+\sum_{i,j=1}^na_{i,j}\frac{\partial^2 \theta^l_{p,q}}{\partial x_i\partial x_j}+<{\nabla\theta^l_{p,q}},b(t,x,\theta^l_{p,q})>+\hat{b}(t,x,\theta^l_{p,q})=0,\mbox{ in } (0,T)\times\Omega_p,\vspace{.1in}\\ 
\theta^l_{p,q}(.,.,a_p)=\theta^l_{p-1,q-1}(.,.,a_p),\mbox{ on } (0,T)\times(-l,l)^{n-1}, \vspace{.1in}\\ 
\theta^l_{p,q}(.,.,b_p)=\theta^l_{p+1,q-1}(.,.,b_p),\mbox{ on } (0,T)\times(-l,l)^{n-1}, \vspace{.1in}\\ 
\theta^l_{p,q}(t,x)=g(t,x),\mbox{ on } (0,T)\times(\partial\mathcal{O}_l\backslash((0,T)\times(-l,l)^{n-1}\times(\{a_p\}\cup\{b_p\}))), \vspace{.1in}\\ 
\theta^l_{p,q}(T,x)=g(T,x),\mbox{ on } (-l,l)^{n-1}\times(a_p,b_p).\end{array}\right. 
\end{equation}
For the extreme subdomain $\Omega_1$ (resp. $\Omega_I$), we consider the boundary condition $\theta^l_{1,q}(t,x,a_1)=g(t,x)$ on the left (resp. $\theta^l_{I,q}(t,x,b_I)=g(t,x)$ on the right) in $(\ref{2e10})$. 
\item Suppose that we stop at the iteration $\#q$ while solving $(\ref{2e10})$. 
\end{itemize}
The following two theorems insist that Step 2 of the algorithm is well-posed and show that the solutions of the subproblems $(\ref{2e10})$ converge to the solution of the main problem $(\ref{2e8})$ when $q$ tends to infinity.
\begin{theorem} 
\label{2t4}
(Well-posedness Theorem) Suppose that $(A1)-(A5)$ hold, then at each iteration $\#q$ in each subdomain $\#p$, there exists a unique classical solution $\theta^l_{p,q}(t,x)$ for $(\ref{2e10})$, such that $\theta^l_{p,q}(t,x)$, $\frac{\partial}{\partial t}\theta^l_{p,q}(t,x)$, $\nabla\theta^l_{p,q}(t,x)$, $\Delta\theta^l_{p,q}(t,x)$ are bounded and the sequence $\{\theta^l_{p,q}\}_{p\in\{1,\dots,I\};q\in\mathbb{N}}$ is uniformly bounded (with respect to $p$ and $q$) in $C((0,T)\times\mathcal{O}_l)$.
\end{theorem}
\begin{theorem}
\label{2t5}
Under Assumptions $(A1)-(A5)$, we have the convergence 
\begin{equation}
\label{2e11}
\mathop{\lim}_{q\to\infty}\mathop{\sup}_{p\in\{1,\dots,I\}}||\theta^l_{p,q}-\theta^l||_{(L^{\infty}((0,T)\times(-l,l)^{n-1}\times(a_p,b_p)))^m}=0.
\end{equation}
\end{theorem}
{\it Step 4:} We will continue with the values $\theta^l_{p,q}$, $p\in\{1,\dots,I\}$ that we have got at the end of step 2.
\begin{itemize}
\item Let $\theta^l_q$ be a function defined on $[0,T]\times\mathcal{O}_l$ such that $\theta^l_q(t,x)=\theta^l_{p,q}(t,x)$ on $[0,T]\times(\Omega_p\backslash(\Omega_{p-1}\cup\Omega_{p+1}))$ for $p\in\{2,\dots,I-1\}$, on $[0,T]\times(\Omega_p\backslash\Omega_{p-1})$ for $p=I$, and on $[0,T]\times(\Omega_p\backslash\Omega_{p+1})$ for $p=1$. We can choose $\theta_q^l(t,x)$ such that it is Lipschitz, differentiable with respect to $x$ and $t$ in $\mathbb{R}^n$ and  $\mathbb{R}$ and
\begin{equation}
\label{2e12}
\mathop{\lim}_{q\to\infty}||\theta^l_{q}-\theta^l||_{(L^{\infty}([0,T]\times\mathcal{O}_l))^m}=0.
\end{equation}
\item Use $\theta^l_q$, solve the following forward SDE
\begin{equation}
\label{2e13}
X_t^q=x+\int_{0}^t\bar{b}_q(s,X_s^q)ds+\int_{0}^t\bar{\sigma}_q(s,X_s^q)d W_s,
\end{equation}
where $\bar{b}_q $ is $b(t,x,\theta^l_q(t,x))$ and $\bar{\sigma}_q(t,x)$ is ${\sigma}(t,x,\theta^l_q(t,x))$. 
\\ Using the same arguments as the ones used for $(\ref{2e5})$, we can conclude that $(\ref{2e13})$ has a unique solution in $\mathbb{L}^2(0,T)$.
\end{itemize}
Set $Y_t^{q,l}$=$\theta_q^l(t,X_t^{q,l})$ and $Z_t^{q,l}=z(t,X_t^{q,l},\theta_q^l(t,X_t^{q,l}),\nabla\theta_q^l(t,X_t^{q,l}))$. The following Theorem says that the sequence $(X_t^{q,l},Y_t^{q,l},Z_t^{q,l})$ converges to $(X_t,$ $Y_t,$ $Z_t)$ as $q$ and $l$ tend to infinity.
\begin{theorem}\label{2t6} (Convergence Theorem)
Suppose that all the assumptions in Section 2.1 hold, then as $q$ and $l$ tend to infinity, $(X_t^{q,l},Y_t^{q,l},Z_t^{q,l})$ converges to the solution $(X_t,Y_t,Z_t)$ of $(\ref{2e1})$ in the following sense
$$\mathop{\lim}_{l\to\infty}\mathop{\lim}_{q\to\infty}\int_0^TE(|X_t^{q,l}-X_t|^2)dt=0,$$
$$\mathop{\lim}_{l\to\infty}\mathop{\lim}_{q\to\infty}\int_0^TE(|Y_t^{q,l}-Y_t|^2)dt=0,$$ 
$$\mathop{\lim}_{l\to\infty}\mathop{\lim}_{q\to\infty}\int_0^TE(|Z_t^{q,l}-Z_t|^\frac{2}{\alpha})dt=0.$$
\end{theorem}
\subsection{Proof of Theorem $\ref{2t4}$ }
\h First of all, we introduce some useful notations which will be needed for the proof. We set $$M_0=\max\{||\theta^l_0||_{C([0,T]\times\mathbb{R}^n)},||g||_{C(\mathbb{R}^n)}\},$$ and define $\rho_{p,q}(t,x)$ to be $\theta^l_{p,q}(T-t,x)$ for $p\in\{1,\dots,I\}$ and $q\in\mathbb{N}$. 
\\ We can reformulate Systems $(\ref{2e10})$ into the following form
\begin{equation}
\label{4e1}
\left \{ \begin{array}{ll}-\frac{\partial \rho_{p,q}}{\partial t}+\sum_{i,j=1}^na_{i,j}\frac{\partial^2 \rho_{p,q}}{\partial x_i\partial x_j}+<{\nabla\rho_{p,q}},b(t,x,\rho_{p,q})>+\hat{b}(t,x,\rho_{p,q})=0,\mbox{ in } (0,T)\times\Omega_p,\vspace{.1in}\\ 
\rho_{p,q}(.,.,a_p)=\rho_{p-1,q-1}(.,.,a_p),\mbox{ on } (0,T)\times(-l,l)^{n-1}, \vspace{.1in}\\ 
\rho_{p,q}(.,.,b_p)=\rho_{p+1,q-1}(.,.,b_p),\mbox{ on } (0,T)\times(-l,l)^{n-1}, \vspace{.1in}\\ 
\rho_{p,q}(t,x)=g(t,x),\mbox{ on } (0,T)\times(\partial\mathcal{O}_l\backslash((-l,l)^{n-1}\times(\{a_p\}\cup\{b_p\}))), \vspace{.1in}\\ 
\rho_{p,q}(0,x)=g(0,x),\mbox{ on } (-l,l)^{n-1}\times(a_p,b_p).\end{array}\right. 
\end{equation}
One can see that $(\ref{4e1})$ are parabolic systems with the initial condition $g$.
\\\h Now, we will prove the Theorem by induction. 
\\\h At step $\#1$, and in the $p$-th subdomain, using the same argument as in Theorem $\ref{2t1}$, we can prove that $(\ref{2e10})$ admits a unique classical solution $\theta^l_{p,1}(t,x)$, where $\theta^l_{p,1}(t,x)$, $\frac{\partial}{\partial t}\theta^l_{p,1}(t,x)$, $\nabla\theta^l_{p,1}(t,x)$, $\Delta\theta^l_{p,1}(t,x)$ are bounded. Consider the following $k$-th equation of $(\ref{4e1})$, for $k$ in $\{1,\dots,m\}$
\begin{equation}
\label{4e2}
\left \{ \begin{array}{ll}-\frac{\partial \rho^k_{p,1}}{\partial t}+\sum_{i,j=1}^na_{i,j}\frac{\partial^2 \rho^k_{p,1}}{\partial x_i\partial x_j}+<{\nabla\rho^k_{p,1}},b(t,x,\rho_{p,1})>+\hat{b}^k(t,x,\rho_{p,1})=0,\mbox{ in } (0,T)\times\Omega_p,\vspace{.1in}\\ 
\rho^k_{p,1}(.,.,a_p)=\rho^k_{p-1,0}(.,.,a_p),\mbox{ on } (0,T)\times(-l,l)^{n-1}, \vspace{.1in}\\ 
\rho^k_{p,1}(.,.,b_p)=\rho^k_{p+1,0}(.,.,b_p),\mbox{ on } (0,T)\times(-l,l)^{n-1}, \vspace{.1in}\\ 
\rho^k_{p,1}(t,x)=g^k(t,x),\mbox{ on } (0,T)\times(\partial\mathcal{O}_l\backslash((-l,l)^{n-1}\times(\{a_p\}\cup\{b_p\}))), \vspace{.1in}\\ 
\rho^k_{p,1}(0,x)=g^k(0,x),\mbox{ on } (-l,l)^{n-1}\times(a_p,b_p).\end{array}\right. 
\end{equation}
Using $(A5)$, we deduce from $(\ref{4e2})$ that
\begin{eqnarray*}
-\frac{\partial \rho^k_{p,1}}{\partial t}+\sum_{i,j=1}^na_{i,j}\frac{\partial^2 \rho^k_{p,1}}{\partial x_i\partial x_j}+<{\nabla\rho^k_{p,1}},b(t,x,\rho_{p,1})>+c(t,x)\rho^k_{p,1}=0,
\end{eqnarray*}
where 
\begin{eqnarray*}
c(t,x)=\left \{ \begin{array}{ll}\frac{\hat{b}^k(t,x,\rho_{p,1})-\hat{b}^k(t,x,\rho^1_{p,1},\dots,\rho^{k-1}_{p,1},0,\rho^{k+1}_{p,1},\dots,\rho^{m}_{p,1})}{\rho^k_{p,1}} &\mbox{ if }\rho^k_{p,1}\ne 0 ,\vspace{.1in}\\ 
0&\mbox{ otherwise }.\end{array}\right. 
\end{eqnarray*}
Since $c(t,x)$ is negative, by applying the maximum principle (see \cite{Evans:1998:PDE}) to this equation, we can see that the maximum and minimum of $\rho^k_{p,1}$ can be obtained on the boundaries, for all $k$ in $\{1,\dots,m\}$. This means that $||\rho_{p,1}||_{C([0,T]\times\Omega_p)}$ is bounded by $M_0$, and then $M_0$ is also an upper bound of $||\theta^l_{p,1}||_{C([0,T]\times\Omega_p)}$.
\\\h Suppose that up to step $q_0$, the unique classical solution $\theta^l_{p,q_0}(t,x)$ exists, $\theta^l_{p,q_0}(t,x)$, $\frac{\partial}{\partial t}\theta^l_{p,q_0}(t,x)$, $\nabla\theta^l_{p,q_0}(t,x)$, $\Delta\theta^l_{p,q_0}(t,x)$ are bounded, and for all $p$ in $\{1,\dots,I\}$, $||\theta^l_{p,q_0}||_{C([0,T]\times\Omega_p)}$ is bounded by $M_0$. We will show that the conclusion is still correct for the step $q_0+1$. The existence and uniqueness of $\theta^l_{p,q_0+1}$ can be infered by using the same argument as in step $\# 1$ and Theorem $\ref{2t1}$. Now, we consider the following equation, for $k$ in $\{1,\dots,m\}$,
\begin{equation}
\label{4e3}
\left \{ \begin{array}{ll}-\frac{\partial \rho^k_{p,q_0+1}}{\partial t}+\sum_{i,j=1}^na_{i,j}\frac{\partial^2 \rho^k_{p,q_0+1}}{\partial x_i\partial x_j}+<{\nabla\rho^k_{p,q_0+1}},b(t,x,\rho_{p,q_0+1})>+\vspace{.1in}\\ 
~~~~~~~~~~~~~~~~~~~~~~~~~~~~~~~~+\hat{b}^k(t,x,\rho_{p,q_0+1})=0,\mbox{ in } (0,T)\times\Omega_p,\vspace{.1in}\\ 
\rho^k_{p,q_0+1}(.,.,a_p)=\rho^k_{p-1,q_0}(.,.,a_p),\mbox{ on } (0,T)\times(-l,l)^{n-1}, \vspace{.1in}\\ 
\rho^k_{p,q_0+1}(.,.,b_p)=\rho^k_{p+1,q_0}(.,.,b_p),\mbox{ on } (0,T)\times(-l,l)^{n-1}, \vspace{.1in}\\ 
\rho^k_{p,q_0+1}(t,x)=g^k(t,x),\mbox{ on } (0,T)\times(\partial\mathcal{O}_l\backslash((0,T)\times(-l,l)^{n-1}\times(\{a_p\}\cup\{b_p\}))), \vspace{.1in}\\ 
\rho^k_{p,q_0+1}(0,x)=g^k(0,x),\mbox{ on } (-l,l)^{n-1}\times(a_p,b_p).\end{array}\right. 
\end{equation}
Again, by a maximum principle argument applied to Equation $(\ref{4e3})$, we can see that the maximum and mininum of $\rho^k_{p,q_0+1}$ can only be obtained on the boundaries, for all $k$ in $\{1,\dots,m\}$. However, we know that $||\theta^l_{p,q_0+1}||_{C([0,T]\times\Omega_p)}$ is bounded by $M_0$ for all $p$ in $\{1,\dots,I\}$, from the definition of $M_0$. We then deduce that $||\theta^l_{p,q_0+1}||_{C([0,T]\times\Omega_p)}$ is bounded by $M_0$. 
\\\h We then conclude that at each step $\#q$ in each subdomain $\#p$, there exists a unique classical solution $\theta_{p,q}(t,x)$ for $(\ref{2e10})$ and the sequence $\{\theta^l_{p,q}\}_{p\in\overline{1,I};q\in\mathbb{N}}$ is uniformly bounded with respect to $p$ and $q$ in $C((0,T)\times\mathcal{O}_l)$ by $M_0$. This concludes the proof.
\subsection{Proof of Theorem $\ref{2t5}$ }
 We divide the proof into 2 steps
\\ {\it Step 1: An exponential decay estimate}
\\\h Setting $e_{p,q}$ to be $\theta^l_{p,q}-\theta^l$, we deduce the system
\begin{equation}
\label{4e4}
\left \{ \begin{array}{ll}\frac{\partial e_{p,q}}{\partial t}+\sum_{i,j=1}^na_{i,j}(t,x,\theta^l_{p,q})\frac{\partial^2 e_{p,q}}{\partial x_i\partial x_j}+<{\nabla e_{p,q}},b(t,x,\theta^l_{p,q})>+c(t,x,\theta^l_{p,q},\theta^l)=0,\vspace{.1in}\\
~~~~~~~~~~~~~~~~~~~~~~~~~~~~~~~~~~~~~~~~~~~~~~~~~~~~~~~~~~~~~~~~~~~~~~~~~~~~~~~\mbox{ in } (0,T)\times\Omega_p,\vspace{.1in}\\ 
e_{p,q}(.,.,a_p)=e_{p-1,q-1}(.,.,a_p),\mbox{ on } (0,T)\times(-l,l)^{n-1}, \vspace{.1in}\\ 
e_{p,q}(.,.,b_p)=e_{p+1,q-1}(.,.,b_p),\mbox{ on } (0,T)\times(-l,l)^{n-1}, \vspace{.1in}\\ 
e_{p,q}(t,x)=0,\mbox{ on } (0,T)\times(\partial\mathcal{O}_l\backslash((0,T)\times(-l,l)^{n-1}\times(\{a_p\}\cup\{b_p\}))), \vspace{.1in}\\ 
e_{p,q}(T,x)=0,\mbox{ on } (-l,l)^{n-1}\times(a_p,b_p),\end{array}\right. 
\end{equation}
where
\begin{eqnarray*}
c(t,x,\theta^l_{p,q},\theta^l)&=&\left[\sum_{i,j=1}^n[a_{i,j}(t,x,\theta^l_{p,q})-a_{i,j}(t,x,\theta^l)]\frac{\partial^2 \theta^l}{\partial x_i\partial x_j}\right]\\
& &+<{\nabla\theta^l},[b(t,x,\theta^l_{p,q})-b(t,x,\theta^l)]>+[\hat{b}(t,x,\theta^l_{p,q})-\hat{b}(t,x,\theta^l)].
\end{eqnarray*}
Now, defining $\epsilon_{p,q}(t,x,y,z)$ to be $e_{p,q}(T-t,x,y,z)$, we change the system into
\begin{equation}
\label{4e5}
\left \{ \begin{array}{ll}\frac{\partial \epsilon_{p,q}}{\partial t}-\sum_{i,j=1}^na_{i,j}(t,x,\theta^l_{p,q})\frac{\partial^2 \epsilon_{p,q}}{\partial x_i\partial x_j}-<{\nabla \epsilon_{p,q}},b(t,x,\theta^l_{p,q})>-c(t,x,\theta^l_{p,q},\theta^l)=0,\vspace{.1in}\\
~~~~~~~~~~~~~~~~~~~~~~~~~~~~~~~~~~~~~~~~~~~~~~~~~~~~~~~~~~~~~~~~~~~~~~~~~~~~~~\mbox{ in } (0,T)\times\Omega_p,\vspace{.1in}\\ 
\epsilon_{p,q}(.,.,a_p)=\epsilon_{p-1,q-1}(.,.,a_p),\mbox{ on } (0,T)\times(-l,l)^{n-1}, \vspace{.1in}\\ 
\epsilon_{p,q}(.,.,b_p)=\epsilon_{p+1,q-1}(.,.,b_p),\mbox{ on } (0,T)\times(-l,l)^{n-1}, \vspace{.1in}\\ 
\epsilon_{p,q}(t,x)=0,\mbox{ on } (0,T)\times(\partial\mathcal{O}_l\backslash((-l,l)^{n-1}\times(\{a_p\}\cup\{b_p\}))), \vspace{.1in}\\ 
\epsilon_{p,q}(0,x)=0,\mbox{ on } (-l,l)^{n-1}\times(a_p,b_p).\end{array}\right. 
\end{equation}
\h We define $$\Phi_{p,q}(t,x)=\sum_{k=1}^m(\epsilon^k_{p,q})^2\exp(\beta(x_n-\omega)-\gamma t),$$ where $\beta$, $\omega$, $\gamma$ will be fixed below, and consider the following parabolic operator
\begin{eqnarray}
\label{4e6}
\begin{split}
\mathfrak{L}(\Phi)=\frac{\partial\Phi}{\partial t}-<\nabla\Phi,b(t,x,\theta^l)>-\sum_{i,j=1}^na_{ij}\frac{\partial^2\Phi}{\partial x_ix_j}+\sum_{i=1}^{n}\beta a_{i,n}\frac{\partial\Phi}{\partial x_i}.
\end{split}
\end{eqnarray}
A direct computation gives
\begin{eqnarray}
\label{4e7}
& &\mathfrak{L}(\Phi_{p,q})\\\nonumber
&=& \sum_{k=1}^m(-\gamma-\beta b^n(t,x,\theta^l_{p,q})+\beta a_{n,n}-\beta^2a_{n,n})(\epsilon^k_{p,q})^2\exp(\beta(x_n-\omega)-\gamma)\\\nonumber
& &+2\epsilon^k_{p,q}\left[-<\nabla\epsilon^k_{p,q},b(t,x,\theta^l_{p,q})>-\sum_{i,j=1}^na_{i,j}\epsilon_{p,q}^k+\frac{\partial\epsilon^k_{p,q}}{\partial t}\right]\\\nonumber
& &-\sum_{i,j=1}^n2a_{i,j}(t,x,\theta^l_{p,q})\frac{\partial\epsilon^k_{p,q}}{\partial x_i}\frac{\partial\epsilon^k_{p,q}}{\partial x_j}\\\nonumber
&\leq& \sum_{k=1}^m\left[(-\gamma-\beta b^n(t,x,\theta^l_{p,q})-\beta^2a_{n,n})(\epsilon^k_{p,q})^2\exp(\beta(x_n-\omega)-\gamma)\right.\\\nonumber
& &\left.+2\epsilon^k_{p,q}c^k(t,x,\theta^l_{p,q},\theta^l)\right],
\end{eqnarray}
where $c^k$ is the $k$-th component of the vector $c$.
\\ We consider the following term of $(\ref{4e7})$
\begin{eqnarray}
\label{4e8}\nonumber
A&=& \sum_{k=1}^m\{(-\gamma-\beta b^n(t,x,\theta^l_{p,q})+\beta a_{n,n}-\beta^2a_{n,n})(\epsilon^k_{p,q})^2+2\epsilon^k_{p,q}c^k(t,x,\theta^l_{p,q},\theta^l)\}\\\nonumber
&=&\sum_{k=1}^m\{(-\gamma-\beta b^n(t,x,\theta^l_{p,q})+\beta a_{n,n}-\beta^2a_{n,n})(\epsilon^k_{p,q})^2\\\nonumber
& &+2\epsilon^k_{p,q} [\sum_{i,j=1}^n[a_{i,j}(t,x,\theta^l_{p,q})-a_{i,j}(t,x,\theta^l)]\frac{\partial^2 \theta^{l,k}}{\partial x_i\partial x_j}\\\nonumber
&+&<{\nabla \theta^{l,k}},[b(t,x,\theta^l_{p,q})-b(t,x,\theta^l)]>+[\hat{b}(t,x,\theta^l_{p,q})-\hat{b}(t,x,\theta^l)]]\}\\\nonumber
&\leq& \{\sum_{k=1}^m(-\gamma-\beta b^n(t,x,\theta^l_{p,q})+\beta a_{n,n}-\beta^2a_{n,n})(\epsilon^k_{p,q})^2\\\nonumber
&+&N_1(\epsilon^k_{p,q})^2||\Delta \theta^l||_{C(\mathbb{R}^n)}+N_2(\epsilon^k_{p,q})^2||\nabla\theta^l||_{C(\mathbb{R}^n)}+N_3(\epsilon^k_{p,q})^2\}\\
&\leq& \sum_{k=1}^m\{(-\gamma-\beta b^n(t,x,\theta^l_{p,q})+\beta a_{n,n}-\beta^2a_{n,n})(\epsilon^k_{p,q})^2+N_4(\epsilon^k_{p,q})^2\},
\end{eqnarray}
where $N_1$, $N_2$, $N_3$, $N_4$ are constants depending only on the coefficients of the system and the bound $M_0$ of $\theta^l_{p,q}$ and $g$ in $C(\mathbb{R}^n)$. Since all solutions of the subproblems $\{\theta^l_{p,q}\}$ are uniformly bounded, $A$ is negative when $\gamma$ is large enough and $\beta$ is suitable chosen. This implies that $\mathfrak{L}(\Phi_{p,q})$ is negative. According to the maximum principle, the maximum of $\Phi_{p,q}$ can only be attained on the boundary of the domain. Which means that the maximum of 
\begin{eqnarray*}\sum_{k=1}^m(e^k_{p,q})^2\exp(\beta(x_n-\omega)-\gamma t)\end{eqnarray*} can only be attained on $\{0\}\times\mathbb{R}^{n-1}\times[a_p,b_p]$, on $(0,T)\times(\partial\mathcal{O}_l\backslash((-l,l)^{n-1}\times(\{a_p\}\cup\{b_p\})))$ or on $([0,T]\times\mathbb{R}^{n-1}\times\{a_p\})\cup ([0,T]\times\mathbb{R}^{n-1}\times\{b_p\})$. 
\\ Since $\Phi_{p,q}(t,x)$ is equal to $0$ on $\{0\}\times\mathbb{R}^{n-1}\times[a_p,b_p]$ and on $(0,T)\times(\partial\mathcal{O}_l\backslash((-l,l)^{n-1}\times(\{a_p\}\cup\{b_p\})))$, we have the following cases:
\\ If $1<p<I$, 
\begin{eqnarray}
\label{4e9}
& &\sum_{k=1}^m(e_{p,q}^k(t,x))^2\exp(\beta(x_n-\omega)-\gamma t)\\\nonumber
&\leq& \max\left\{\max_{(t,x)\in[0,T]\times[-l,l]^{n-1}\times\{a_p\}}\sum_{k=1}^m(e_{p,q}^k(t,x))^2\exp(\beta(a_p-\omega)-\gamma t)\right.,\\\nonumber
& &\left.\max_{(t,x)\in[0,T]\times[-l,l]^{n-1}\times\{b_p\}}\sum_{k=1}^m(e_{p,q}^k(t,x))^2\exp(\beta(b_p-\omega)-\gamma t)\right\}.
\end{eqnarray}
If $p=1$
\begin{eqnarray}
\label{4e10}
& &\sum_{k=1}^m(e_{1,q}^k(t,x))^2\exp(\beta(x_n-\omega)-\gamma t)\\\nonumber
&\leq& \max_{(t,x)\in[0,T]\times[-l,l]^{n-1}\times\{b_1\}}\sum_{k=1}^m(e_{1,q}^k(t,x))^2\exp(\beta(b_1-\omega)-\gamma t).
\end{eqnarray}
\\ If $p=I$
\begin{eqnarray}
\label{4e11}
& &\sum_{k=1}^m(e_{I,q}^k(t,x))^2\exp(\beta(x_n-\omega)-\gamma t)\\\nonumber
&\leq& \max_{(t,x)\in[0,T]\times[-l,l]^{n-1}\times\{a_I\}}\sum_{k=1}^m(e_{I,q}^k(t,x))^2\exp(\beta(a_I-\omega)-\gamma t).
\end{eqnarray}
\\ {\it Step 2: Proof of the Convergence}
{\\\h\bf Step 2.1:} Estimate of the right boudaries of the sub-domains.
\\\h For $x$ in $[-l,l]^n$, we denote $x$ by $(X,x_n)$, where $X\in[-l,l]^{n-1}$ and $x_n\in[-l,l]$. Moreover, we define
$$E_q=\max_{p\in\{1,\dots,I\}}\left\{\max_{(t,x)\in[0,T]\times[-l,l]^n}\sum_{k=1}^m(e_{p,q}^k(t,x))^2\exp(-\gamma t)\right\}.$$
\\ Consider the $I$-th domain, at the $q$-th step, we can see that $(\ref{4e11})$ infers
\begin{eqnarray*}
& &\sum_{k=1}^m(e_{I,q}^k(t,X,x_n))^2\exp(\beta(x_n-a_I)-\gamma t)\\\nonumber
&\leq& \max_{(t,X)\in[0,T]\times[-l,l]^{n-1}}\sum_{k=1}^m(e_{I,q}^k(t,X,a_I))^2\exp(-\gamma t),
\end{eqnarray*}
where $\omega$ is replaced by $a_I$.
\\ Replacing $x_n$ by $b_{I-1}$ in the previous inequality, we obtain
\begin{eqnarray*}
& &\sum_{k=1}^m(e_{I,q}^k(t,X,b_{I-1}))^2\exp(\beta(b_{I-1}-a_I)-\gamma t)\\\nonumber
&\leq& \max_{(t,X)\in[0,T]\times[-l,l]^{n-1}}\sum_{k=1}^m(e_{I,q}^k(t,X,a_I))^2\exp(-\gamma t).
\end{eqnarray*}
Since $e_{I,q}^k(t,X,b_{I-1})$ is equal to $e_{I-1,q+1}^k(t,X,b_{I-1})$, then
\begin{eqnarray*}
& &\sum_{k=1}^m(e_{I-1,q+1}^k(t,X,b_{I-1}))^2\exp(\beta(b_{I-1}-a_I)-\gamma t)\\\nonumber
&\leq& \max_{(t,X)\in[0,T]\times[-l,l]^{n-1}}\sum_{k=1}^m(e_{I,q}^k(t,X,a_I))^2\exp(-\gamma t).
\end{eqnarray*}
We define $\beta_1$ to be $\sqrt\frac{\gamma}{2}$ and let $\beta$ in this case be $\beta_1$; then if we choose $\gamma$ large, $\gamma-\beta^2$ is large, the inequality becomes
\begin{eqnarray*}
& &\sum_{k=1}^m(e_{I-1,q+1}^k(t,X,b_{I-1}))^2\exp(-\gamma t)\\\nonumber
&\leq& \exp(-\beta_1 S_{I-1})\max_{(t,X)\in[0,T]\times[-l,l]^{n-1}}\sum_{k=1}^m(e_{I,q}^k(t,X,a_I))^2\exp(-\gamma t).
\end{eqnarray*}
We deduce that
\begin{equation}
\label{4e12}
\sum_{k=1}^m(e_{I-1,q+1}^k(t,X,b_{I-1}))^2\exp(-\gamma t)\leq \exp(-\beta_1S_{I-1})E_q.
\end{equation}
\h Moreover, on the $(I-1)$-th domain, at the $(q+1)$-th step, $(\ref{4e9})$ leads to
\begin{eqnarray*}
& &\sum_{k=1}^m(e_{I-1,q+1}^k(t,X,x_n))^2\exp(\beta(x_n-a_{I-1})-\gamma t)\\\nonumber
&\leq& \max\left\{\max_{(t,X)\in[0,T]\times[-l,l]^{n-1}}\sum_{k=1}^m(e_{I-1,q+1}^k(t,X,a_{I-1}))^2\exp(-\gamma t),\right.\\\nonumber
& &\left.\max_{(t,X)\in[0,T]\times[-l,l]^{n-1}}\sum_{k=1}^m(e_{I-1,q+1}^k(t,X,b_{I-1}))^2\exp(\beta(b_{I-1}-a_{I-1})-\gamma t)\right\},
\end{eqnarray*}
where $\omega$ is replaced by $a_{I-1}$.
\\ Since $e_{I-1,q+1}^k(t,X,b_{I-2})$ is equal to $e_{I-2,q+2}^k(t,X,b_{I-2})$, then
\begin{eqnarray*}
& &\sum_{k=1}^m(e_{I-2,q+2}^k(t,X,b_{I-2}))^2\exp(\beta(b_{I-2}-a_{I-1})-\gamma t)\\\nonumber
&\leq& \max\left\{\max_{(t,X)\in[0,T]\times[-l,l]^{n-1}}\sum_{k=1}^m(e_{I-1,q+1}^k(t,X,a_{I-1}))^2\exp(-\gamma t),\right.\\\nonumber
& &\left.\max_{(t,X)\in[0,T]\times[-l,l]^{n-1}}\sum_{k=1}^m(e_{I-1,q+1}^k(t,X,b_{I-1}))^2\exp(\beta L_{I-1}-\gamma t)\right\};
\end{eqnarray*}
thus
\begin{eqnarray*}
& &\sum_{k=1}^m(e_{I-2,q+2}^k(t,X,b_{I-2}))^2\exp(\beta S_{I-2}-\gamma t)\\\nonumber
&\leq& \max\left\{\max_{(t,X)\in[0,T]\times[-l,l]^{n-1}}\sum_{k=1}^m(e_{I-1,q+1}^k(t,X,a_{I-1}))^2\exp(-\gamma t),\right.\\\nonumber
& &\left.\max_{(t,X)\in[0,T]\times[-l,l]^{n-1}}\sum_{k=1}^m(e_{I-1,q+1}^k(t,X,b_{I-1}))^2\exp(\beta L_{I-1}-\gamma t)\right\}.
\end{eqnarray*}
Combining this with $(\ref{4e12})$ and the fact that $$\max_{(t,X)\in[0,T]\times[-l,l]^{n-1}}\sum_{k=1}^m(e_{I-1,q+1}^k(t,X,a_{I-1}))^2\exp(-\gamma t)\leq E_{q+1},$$ we obtain that   
$$\sum_{k=1}^m(e_{I-2,q+2}^k(t,X,b_{I-2}))^2\exp(\beta S_{I-2}-\gamma t)\leq\max\{E_{q}\exp(\beta L_{I-1}-\beta_1S_{I-1}),E_{q+1}\}.$$
Thus
\begin{eqnarray*}
\sum_{k=1}^m(e_{I-2,q+2}^k(t,X,b_{I-2}))^2\exp(-\gamma t)& \leq&\max\{E_{q}\exp(\beta (L_{I-1}-S_{I-2})-\beta_1S_{I-1}),\\
& &E_{q+1}\exp(-\beta S_{I-2})\}.
\end{eqnarray*}
Defining $\beta_2$ to be $\beta_1\frac{S_{I-1}}{L_{I-1}}$ and choosing $\beta$ to be $\beta_2$ such that $$\beta_2(-L_{I-1}+S_{I-2})+\beta_1S_{I-1}=\beta_2S_{I-2},$$ we infer 
\begin{equation}
\label{4e13}
\sum_{k=1}^m(e_{I-2,q+2}^k(t,X,b_{I-2}))^2\exp(-\gamma t)\leq\max\{E_{k},E_{k+1}\}\exp(-\beta_2S_{I-2}).
\end{equation}
\h Using the same techniques as the ones that we use to achive $(\ref{4e12})$ and $(\ref{4e13})$, we can prove that
\begin{equation}
\label{4e14}
\sum_{k=1}^m(e_{I-j,q+j}^k(t,X,b_{I-j}))^2\exp(-\gamma t)\leq\max\{E_{k},\dots,E_{k+j-1}\}\exp(-\beta_jS_{I-j}), 
\end{equation}
where $$\beta_j=\beta_1\frac{S_{I-1}}{L_{I-1}}\dots\frac{S_{I-j+1}}{L_{I-j+1}},~~j\in\{2,\dots,I-1\}.$$
\\ {\h\bf Step 2.2:} Estimate of the left boudaries of the sub-domains
\\\h Consider the $1$-th domain, at the $k$-th step. Then $(\ref{4e10})$ infers that
\begin{eqnarray*}
& &\sum_{k=1}^m(e_{1,q}^k(t,X,x_n))^2\exp(\beta(x_n-b_1)-\gamma t)\\\nonumber
&\leq& \max_{(t,X)\in[0,T]\times[-l,l]^{n-1}}\sum_{k=1}^m(e_{1,q}^k(t,X,b_1))^2\exp(-\gamma t),
\end{eqnarray*}
we notice here that $\omega$ is replaced by $b_1$.
\\ Replace $x_n$ by $a_{2}$, we obtain that
\begin{eqnarray*}
& &\sum_{k=1}^m(e_{1,q}^k(t,X,a_2))^2\exp(\beta(a_2-b_1)-\gamma t)\\\nonumber
&\leq& \max_{(t,X)\in[0,T]\times[-l,l]^{n-1}}\sum_{k=1}^m(e_{1,q}^k(t,X,b_1))^2\exp(-\gamma t),
\end{eqnarray*}
Since $e_{1,q}^k(t,X,a_2)$ is equal to $e_{2,q+1}^k(t,X,a_2)$,
\begin{eqnarray*}
& &\sum_{k=1}^m(e_{2,q+1}^k(t,X,a_2))^2\exp(-\gamma t)\\\nonumber
&\leq& \max_{(t,X)\in[0,T]\times[-l,l]^{n-1}}\sum_{k=1}^m(e_{1,q}^k(t,X,b_1))^2\exp(-\gamma t)\exp(\beta S_{1}),
\end{eqnarray*} 
We define $\beta_1'$ to be $-\sqrt\frac{\gamma}{2}$ and let $\beta$ be $\beta_1'$ in this case. If we choose $\gamma$ large, $\gamma-\beta^2$ is large. The inequality becomes
\begin{eqnarray*}
& &\sum_{k=1}^m(e_{2,q+1}^k(t,X,a_2))^2\exp(-\gamma t)\\\nonumber
&\leq& \exp(-\beta_1' S_{1})\max_{(t,X)\in[0,T]\times[-l,l]^{n-1}}\sum_{k=1}^m(e_{1,q}^k(t,X,b_1))^2\exp(-\gamma t).
\end{eqnarray*} 
We deduce
\begin{equation}
\label{4e15}
\sum_{k=1}^m(e_{2,q+1}^k(t,X,a_2))^2\exp(-\gamma t)\leq \exp(-\beta'_1S_1)E_q.
\end{equation}
\h Moreover, on the $2$-th domain, at the $(q+1)$-th step, $(\ref{4e9})$ leads to
\begin{eqnarray*}
& &\sum_{k=1}^m(e_{2,q+1}^k(t,X,x_n))^2\exp(\beta(x_n-a_2)-\gamma t)\\\nonumber
&\leq& \max\left\{\max_{(t,X)\in[0,T]\times[-l,l]^{n-1}}\sum_{k=1}^m(e_{2,q+1}^k(t,X,a_2))^2\exp(-\gamma t),\right.\\\nonumber
& &\left.\max_{(t,X)\in[0,T]\times[-l,l]^{n-1}}\sum_{k=1}^m(e_{2,q+1}^k(t,X,b_2))^2\exp(\beta(b_2-a_2)-\gamma t)\right\},
\end{eqnarray*}
notice that $\omega$ is replaced by $a_2$.
\\ Since $e_{2,q+1}^k(t,X,a_3)$ is equal to $e_{3,q+2}^k(t,X,a_3)$, then
\begin{eqnarray*}
& &\sum_{k=1}^m(e_{3,q+2}^k(t,X,a_3))^2\exp(\beta(a_3-a_2)-\gamma t)\\\nonumber
&\leq& \max\left\{\max_{(t,X)\in[0,T]\times[-l,l]^{n-1}}\sum_{k=1}^m(e_{2,q+1}^k(t,X,a_2))^2\exp(-\gamma t)\right.,\\\nonumber
& &\left.\max_{(t,X)\in[0,T]\times[-l,l]^{n-1}}\sum_{k=1}^m(e_{2,q+1}^k(t,X,b_2))^2\exp(\beta(b_2-a_2)-\gamma t)\right\}.
\end{eqnarray*}
Hence
\begin{eqnarray*}
& &\sum_{k=1}^m(e_{3,q+2}^k(t,X,a_3))^2\exp(-\gamma t)\\\nonumber
&\leq&\exp(-\beta(L_2-S_2)) \max\left\{\max_{(t,X)\in[0,T]\times[-l,l]^{n-1}}\sum_{k=1}^m(e_{2,q+1}^k(t,X,a_2))^2\exp(-\gamma t),\right.\\\nonumber
& &\left.\max_{(t,X)\in[0,T]\times[-l,l]^{n-1}}\sum_{k=1}^m(e_{2,q+1}^k(t,X,b_2))^2\exp(\beta L_2-\gamma t)\right\}.
\end{eqnarray*}
Combining this inequality, $(\ref{4e15})$ and the fact that $$\sum_{k=1}^m(e_{2,q+1}^k(t,X,b_2))^2\exp(-\gamma t)\leq E_{q+1},$$  we deduce   
\begin{eqnarray*}
& &\sum_{k=1}^m(e_{3,q+2}^k(t,X,a_3))^2\exp(-\gamma t)\\\nonumber
&\leq&\exp(-\beta(L_2-S_2)) \max\{\exp(-\beta'_1S_1)E_q,\exp(\beta L_2)E_{q+1}\},
\end{eqnarray*}
Define $\beta_2'$ to be $\beta'_1\frac{S_1}{L_{2}}$ and choose $\beta$ to be $-\beta_2'$, we  infer that $$-\beta(L_2-S_2)-\beta'_1S_1=\beta S_2.$$ This implies
\begin{equation}
\label{4e16}
\sum_{k=1}^m(e_{3,q+2}^k(t,X,a_3))^2\exp(-\gamma t)\leq\max\{E_{k},E_{k+1}\}\exp(-\beta'_2S_2).
\end{equation}
\h Using the same techniques as the ones that we use to achive $(\ref{4e15})$ and $(\ref{4e16})$, we can prove that
\begin{equation}
\sum_{k=1}^m(e_{j,q+j-1}^k(t,X,a_j))^2\exp(-\gamma t)\leq\max\{E_{k},\dots,E_{k+j-2}\}\exp(-\beta_{j-1}'S_{j-1}), 
\label{4e17}
\end{equation}
where $$\beta_j'=\beta_1'\frac{S_1}{L_{2}}\dots\frac{S_{j-1}}{L_{j}}, ~~j\in\{2,\dots,I-1\}.$$
\\ {\h\bf Step 2.3:} Convergence result
\\\h Setting $$\bar\epsilon=\sqrt \frac{\gamma}{2}\frac{S_1\dots S_{I-1}}{L_2\dots L_{I-1}},$$ and
$$\bar E_k=\max_{j\in\{0,\dots,I-1\}}\{E_{k+j}\},$$
we infer from $(\ref{4e14})$ and $(\ref{4e17})$ that
$$\bar E_{k+1}\leq \bar E_k\exp(-\bar\epsilon),\forall k\in\mathbb{N}.$$
Therefore
$$\bar E_{n}\leq \bar E_0\exp(-n\bar\epsilon),\forall n\in\mathbb{N},$$
\\\h Hence $E_k$ tends to $0$ as $k$ tends to infinity. Which gives that
$$\lim_{q\to\infty}\max_{p=\overline{1,I}}\sum_{k=1}^m||e_{p,q}^k||_{C([0,T]\times\mathbb{R}^{n-1}\times[a_p,b_p])}=0,$$
and this concludes the proof.
\subsection{Proof of Theorem $\ref{2t6}$}
\h We start proving that $$\mathop{\lim}_{l\to\infty}\mathop{\lim}_{q\to\infty}\int_0^TE(|X_t^{q,l}-X_t|^2)dt=0.$$
\\ Subtracting $(\ref{2e9})$ and $(\ref{2e13})$, we get
\begin{eqnarray}
\label{4e18}
X_t-X_t^{q,l}&=&\int_0^t[b(s,X_s,\theta(s,X_s))-b(s,X_s^{q,l},\theta^{q,l}(s,X^{q,l}))]ds\\\nonumber
& &+\int_0^t[\sigma((s,X_s,\theta(s,X_s))-\sigma(s,X_s^{q,l},\theta^{q,l}(s,X^{q,l}_s))]dW_s,
\end{eqnarray}
which leads to
\begin{eqnarray}
\label{4e19}
|X_t-X_t^{q,l}|^2&\leq&2\left(\int_0^t[b(s,X_s,\theta(s,X_s))-b(s,X_s^{q,l},\theta^{q,l}(s,X^{q,l}_s))]ds\right)^2\\\nonumber
& &+2\left(\int_0^t[\sigma(s,X_s,\theta(s,X_s))-\sigma(s,X_s^{q,l},\theta^{q,l}(s,X^{q,l}_s))]dW_s\right)^2.
\end{eqnarray}
A simple calculation gives
\begin{eqnarray}
\label{4e20}
E(|X_t-X_t^{q,l}|^2)&\leq&N_5E\int_0^t|b(s,X_s,\theta(s,X_s))-b(s,X_s^{q,l},\theta^{q,l}(s,X^{q,l}_s))|^2ds\\\nonumber
& &+N_5E\int_0^t|\sigma(s,X_s,\theta(s,X_s))-\sigma(s,X_s^{q,l},\theta^{q,l}(s,X^{q,l}_s))|^2ds\\\nonumber
&\leq&N_6E\int_0^t[|X_s-X_s^{q,l}|^2+|\theta(s,X_s)-\theta^{q,l}(s,X_s^{q,l})|^2]ds\\\nonumber
&\leq&N_6E\int_0^t|X_s-X_s^{q,l}|^2ds\\\nonumber
& +&N_7E\int_0^t[|\theta(s,X_s^{q,l})-\theta^{q,l}(s,X_s^{q,l})|^2+|\theta(s,X_s)-\theta(s,X_s^{q,l})|^2]ds\\\nonumber
&\leq&N_8E\int_0^t[|X_s-X_s^{q,l}|^2+|\theta(s,X_s^{q,l})-\theta^{q,l}(s,X_s^{q,l})|^2]ds,
\end{eqnarray}
where $N_5$, $N_6$, $N_7$, $N_8$ are positive constants.
\\ Since $\{\theta^{q,l}\}$ converges uniformly to $\theta^l$, and $\{\theta^{l}\}$ converges uniformly to $\theta$, then for all positive number $\epsilon$, there exists $Q(\epsilon)$ such that
 $$|\theta(s,X_s^{q,l})-\theta^{q,l}(s,X_s^{q,l})|<\sqrt\epsilon, \forall q,l>Q(\epsilon).$$
 Inequality $(\ref{4e20})$ leads to
\begin{eqnarray}
\label{4e21}
E(|X_t-X_t^{q,l}|^2)&\leq&N_8E\int_0^t[|X_s-X_s^{q,l}|^2+\epsilon]ds,~~\forall q,l>Q(\epsilon).
\end{eqnarray}
Now, setting $$H(t)=\int_0^tE(|X_s-X_s^{q,l}|^2)ds,$$ we obtain 
\begin{eqnarray}
\label{4e22}
H'(t)\leq N_9H(t)+N_9\epsilon,
\end{eqnarray}
where $N_9$ is a positive constant.
Therefore $$H'(t)\exp(-N_9t)-N_9H(t)\exp(-N_9t)- N_9\epsilon\exp(-N_9t)\leq 0.$$ This implies $$(H(t)\exp(-N_9t)+\epsilon\exp(-N_9t))'\leq 0,$$ and this inequality then leads to $$H(t)\exp(-N_9t)+\epsilon\exp(-N_9t)\leq \epsilon.$$ Consequently, $$H(t)\leq\epsilon(\exp(N_9T)-1)$$ for $q,l$ greater than $Q(\epsilon)$, which leads to $$\mathop{\lim}_{l\to\infty}\mathop{\lim}_{q\to\infty}\int_0^TE(|X_t-X_t^{q,l}|^2)dt=0.$$
\\\h We next prove that $$\mathop{\lim}_{l\to\infty}\mathop{\lim}_{q\to\infty}\int_0^TE(|Y_t^{q,l}-Y_t|^2)dt=0.$$
\\ The fact that $$\mathop{\lim}_{l\to\infty}\mathop{\lim}_{q\to\infty}\int_0^TE(|X_t-X_t^{q,l}|^2)dt=0$$ implies that $X^{q,l}_t$ converges to $X_t$ almost everywhere. From which, we can infer that $\theta^{q,l}(t,X^{q,l}_t)$ converges to $\theta(t,X_t)$ almost everywhere. Since $\{\theta^{q,l}(t,X^{q,l}_t)\}$ is bounded, the Lebesgue Dominated Convergence Theorem gives that $$\mathop{\lim}_{l\to\infty}\mathop{\lim}_{q\to\infty}E\int_t^T|Y_s^{q,l}-Y_s|^2ds=0,$$ and then $$\mathop{\lim}_{l\to\infty}\mathop{\lim}_{q\to\infty}\int_0^TE(|Y_t^{q,l}-Y_t|^2)dt=0.$$
\\\h Now, we prove that $$\mathop{\lim}_{l\to\infty}\mathop{\lim}_{q\to\infty}E\int_0^T|Z_t^{q,l}-Z_t|^\frac{2}{\alpha}dt=0.$$ 
\\ We get the following equation from the definition of $Y_t$ and $Y_t^{q,l}$,
\begin{eqnarray}
\label{4e23}
Y_t-Y_t^{q,l}&=&\int_t^T[\hat{b}(s,X_s,Y_s)-\hat{b}(s,X_s^{q,l},Y^{q,l}_s)]ds\\\nonumber
& &+\int_t^T[\hat{\sigma}(s,X_s,Y_s,Z_s)-\hat{\sigma}(s,X_s^{q,l},Y_s^{q,l},Z_s^{q,l})]dW_s.
\end{eqnarray}
This implies that
\begin{eqnarray}
\label{4e24}
& & E\int_t^T[\hat{\sigma}(s,X_s,Y_s,Z_s)-\hat{\sigma}(s,X_s,Y_s,Z_s^{q,l})]^2ds\\\nonumber
&\leq&N_{10}\left(E(|Y_t-Y_t^{q,l}|^2)+E\int_t^T[\hat{b}(s,X_s,Y_s)-\hat{b}(s,X_s^{q,l},Y^{q,l}_s)]^2ds\right)\\\nonumber
&\leq&N_{11}\left(E(|Y_t-Y_t^{q,l}|^2)+E\int_t^T[|X_s-X_s^{q,l}|^2+|Y_s-Y_s^{q,l}|^2]ds\right),
\end{eqnarray}
where $N_{10}$, $N_{11}$ are positive constants.
\\ Applying Condition $(A3)$ to the inequality $(\ref{4e24})$, we get
 \begin{eqnarray}
\label{4e25}
& & E\int_t^T|Z_s-Z_s^{q,l}|^\frac{2}{\alpha}ds\\\nonumber
&\leq&N_{12}\left(E(|Y_t-Y_t^{q,l}|^2)+E\int_t^T[|X_s-X_s^{q,l}|^2+|Y_s-Y_s^{q,l}|^2]ds\right),
\end{eqnarray}
where $N_{12}$ is a positive constant.
\\ Due to the fact that $$\mathop{\lim}_{l\to\infty}\mathop{\lim}_{q\to\infty}\int_0^TE(|X_t-X_t^{q,l}|^2)dt=0,$$ and $$\mathop{\lim}_{l\to\infty}\mathop{\lim}_{q\to\infty}E\int_0^T|Y_t-Y_t^{q,l}|^2dt=0,$$ we can see that the RHS of $(\ref{4e25})$ tends to $0$ as $q$ and $l$ tend to infinity. Hence $E\int_0^T|Z_t-Z_t^{q,l}|^\frac{2}{\alpha}dt$ tends to $0$ as $q$ and $l$ tend to infinity. This concludes the proof.
\section{Conclusion}
We have introduced a new Domain Decomposition Method for a system of SDEs. The method has been studied theoretically and proved to be well-posed and stable. We have also proposed a new technique to prove the convergence of Domain Decomposition Methods for systems of nonlinear parabolic equations in $n$-dimension. The method has the potential to be used to prove the convergence of Domain Decomposition Methods for many kinds of nonlinear problems.
\\ {\bf Acknowledgements.} The author is grateful to Professor Laurence Halpern for her encouragement. He would also like to thank Dr Eulalia Nualart for her kindness and careful reading of the manuscript.
\bibliographystyle{plain}\bibliography{DDMforFBSDE}
\end{document}